\documentclass[a4paper,12pt,twoside]{article}

\usepackage{amsmath,amsthm,amscd,amsfonts,amssymb,amsxtra}
\usepackage[hmargin=1.2in,vmargin=1.2in]{geometry}
\usepackage[utf8]{inputenc}
\usepackage{graphicx}
\usepackage{microtype}
\usepackage[pdftex,bookmarks,colorlinks=false]{hyperref}
\usepackage{booktabs,enumitem,caption,subcaption}
\usepackage[mathscr]{euscript}
\usepackage{fleqn}

\usepackage[auth-sc-lg,affil-it]{authblk}
\usepackage{setspace}
\numberwithin{equation}{section}

\newtheorem{thm}{Theorem}[section]
\newtheorem{defn}{Definition}[section]
\newtheorem{lem}[thm]{Lemma}
\newtheorem{prop}[thm]{Proposition}
\newtheorem{cor}[thm]{Corollary}

\def\ni{\noindent}
\def\N{\mathbb{N}}
\def\R{\mathbb{R}}

\title{\textbf{\sc On Chromatic Zagreb Indices of Certain Graphs}}

\author{Johan Kok}
\affil{\small Tshwane Metropolitan Police Department\\ City of Tshwane, Republic of South Africa \\ {\tt kokkiek2@tshwane.gov.za}}

\author{N. K. Sudev}
\affil{\small Department of Mathematics\\ Vidya Academy of Science \& Technology \\ Thrissur - 680501, India.\\ {\tt sudevnk@gmail.com}}

\author{U. Mary}
\affil{\small Department of Mathematics \\ Nirmala College for Women\\ Coimbatore, India.\\ {\tt marycbe@gmail.com}}

\date{}

\begin{document}
\maketitle

\begin{abstract}
Let $G$ be a finite and simple undirected connected graph of order $n\ge 1$ and let $\varphi$ be a proper vertex colouring of $G$. Denote $\varphi:v_i \mapsto c_j$ simply, $c(v_i) = c_j$.  In this paper, we introduce a variation of the well-known Zagreb indices by utilising the parameter $c(v)$ instead of the invariant $d(v)$ for all vertices of $G$. The new indices are called chromatic Zagreb indices. We study these new indices for certain classes of graphs and introduce the notion of chromatically stable graphs. 
\end{abstract}

\ni  \textbf{Keywords:} Zagreb index, chromatic Zagreb index, tree, caterpillar, thorn graph, chromatically stable graphs.
\vspace{0.2cm}

\ni \textbf{Mathematics Subject Classification:} 05C15, 05C55, 05D40.

\section{Introduction}

For general notation and concepts in graph and digraph theory, we refer to \cite{BM1,BLS,CL1,FH1, DBW} and for the terminology of graph colouring, we refer to \cite{CZ1,JT1,MK1}. Unless mentioned otherwise, a graph will be a finite and simple undirected  connected graph $G$ of order $n$. In this paper, trees are considered as a specialised subset of general simple connected graphs. 

Several structural properties of graphs such as its size (the number of edges), orientation, degree of vertices, minimum and maximum degree of the graphs etc. are studied in an extensive manner, in the context of many graph theoretic concepts and a vast amount of results and applications thereof are found in the literature.

The topological graph indices related to irregularity of a graph, namely the \textit{first Zagreb index} $M_1(G)$ and the \textit{second Zagreb index} $M_2(G)$, are the oldest irregularity measures studied. A new \textit{irregularity} of $G$ has been introduced in \cite{MOA1} as $irr(G)=\sum\limits_{uv\in E(G)}imb(uv),\ imb(uv)=|d(v)-d(u)|$.  In \cite{GHF1}, this new index was named the \textit{third Zagreb index} to conform with the terminology of chemical graph theory. 

Recently, another topological index called \textit{total irregularity} of graphs has been introduced in \cite{ABD1} as $irr_t(G)=\frac{1}{2}\sum\limits_{u,v \in V(G)}|d(u)-d(v)|$.  This parameter may be called the \textit{fourth Zagreb index}. The study of Zagreb indices is strongly dependent on the structural property of edges of the graph $G$ under consideration and subsequently the degree of the vertices of $G$. An interesting study on \textit{Zagreb coindices} has been done in \cite{TD1} and later an extended study on Zagreb coindices of graph operations has been done in \cite{ADH1}. In a sense, paper \cite{TD1} may be considered to be the initial study on irregularity on an \textit{absent} property which resembles the study of \textit{structural virtuality}. 

With the advancement of ICT derivatives and the creation of several virtual structures, studies on the graphs with virtual properties is becoming more and more important. Some relevant examples of real and salient properties, almost physical-virtual properties, are the magnetic field around electrically charged objects, gravitational force between objects with mass and radiation index of radio-active elements or heated elements. Replicated research which takes into account of parametrised properties could provide some insight into technology allocation to the vertices of a graph. The earliest studies in this direction would be the study of weighted graphs, graph labeling, graph colouring, signed graphs, set-labeling and others. 

We recall that if $\mathcal{C}= \{c_1,c_2,c_3,\dots,c_\ell\}$ is a set of distinct colours (or labels or weights), a \textit{proper vertex colouring} of a graph $G$, denoted $\varphi:V(G) \mapsto \mathcal{C}$, is a vertex colouring of $G$ such that no two distinct adjacent vertices of $G$ have the same colour. The set of all vertices of $G$ which have the colour $c_i$ is called the \textit{colour class} of $c_i$ in $G$. The cardinality of the colour class of a colour $c_i$ is said to be the \textit{strength} of that colour in $G$ and is denoted by $\theta(c_i)$.  

The cardinality of a minimum set of colours which allows a proper vertex colouring of $G$ is called the \textit{chromatic number} of $G$ and is denoted $\chi(G)$.  When a vertex colouring is considered with colours of minimum subscripts the colouring is called a \textit{minimum parameter} colouring. Unless stated otherwise, we consider minimum parameter colour sets throughout in this paper. 

\ni We also recall the first three Zagreb indices which are defined as follows (see \cite{GT1,BZ1,ZG1} for the notion of first two Zagreb indices and see \cite{ABD1,GHF1} for third Zagreb index).

\vspace{-0.75cm}

\begin{eqnarray*}
M_1(G) & = & \sum\limits_{i=1}^{n}d^2(v_i) = \sum \limits_{i=1}^{n-1}\sum\limits_{j=2}^{n} (d(v_i) + d(v_j));\ v_iv_j \in E(G),\\
M_2(G) & = & \sum \limits_{i=1}^{n-1}\sum\limits_{j=2}^{n} d(v_i)d(v_j);\ v_iv_j \in E(G),\\
M_3(G) & = & \sum \limits_{i=1}^{n-1}\sum\limits_{j=2}^{n}|d(v_i) - d(v_j)|;\ v_iv_j \in E(G).
\end{eqnarray*}

We define $M_1(K_1)= M_2(K_1)=M_3(K_1)=0$ as the default values. Note that the Zagreb indices are all functions of the degree of vertices of graph $G$.  For a given graph, the vertex degrees are invariants to determine the above mentioned Zagreb indices. 

Now, we proceed to define the colouring analogue of different Zagreb indices, namely  \textit{chromatic Zagreb indices}, by replacing the invariants $d(v_i)$ by the parameters $s$, where $c(v_i) = c_s,\ \forall\, v_i \in V(G)$. First, note that for any minimum parameter set of colours $\mathcal{C}$, $|\mathcal{C}| = \ell$, a graph $G$ has $\ell!$ minimum parameter colourings. Denote these colourings $\varphi_t(G)$, $1\le t \le \ell!$. Now, the variable chromatic Zagreb indices  can be defined as follows.

\vspace{-0.75cm}

\begin{eqnarray*}
M^{\varphi_t}_1(G) & = & \sum\limits_{i=1}^{n}s^2,\ c(v_i)=c_s,\\
&= & \sum\limits_{j=1}^{\ell}\theta(c_j)\cdot j^2,\ c_j \in \mathcal{C}, \\
M^{\varphi_t}_2(G)& = &\sum\limits_{i=1}^{n-1}\sum\limits_{j=2}^{n}(s\cdot k),\ v_iv_j \in E(G),\ c(v_i)=s,\ c(v_j)=k,\\
M^{\varphi_t}_3(G) & = & \sum\limits_{i=1}^{n-1}\sum\limits_{j=2}^{n}|s-k|,\ v_iv_j \in E(G),\  c(v_i)=s,\ c(v_j)=k,  
\end{eqnarray*}
where $1\le t \le \ell\,!$.

In view of the above notions, we define the minimum and maximum chromatic Zagreb indices as follows.

\vspace{-0.75cm}

\begin{eqnarray*}
M^{\varphi^-}_1(G) & = & \min\{M^{\varphi_t}_1(G): 1\le t \le \ell!\}, \\
M^{\varphi^+}_1(G) & = & \max\{M^{\varphi_t}_1(G): 1\le t \le \ell!\}, \\
M^{\varphi^-}_2(G) & = & \min\{M^{\varphi_t}_2(G): 1\le t \le \ell!\}, \\
M^{\varphi^+}_2(G) & = & \max\{M^{\varphi_t}_2(G): 1\le t \le \ell!\}, \\
M^{\varphi^-}_3(G) & = & \min\{M^{\varphi_t}_3(G): 1\le t \le \ell!\}, \\
M^{\varphi^+}_3(G) & = & \max\{M^{\varphi_t}_3(G): 1\le t \le \ell!\}.
\end{eqnarray*}
 
We also define $M^{\varphi^-}_2(K_1)= M^{\varphi^+}_2(K_1)=0$ and $M^{\varphi^-}_3(K_1)=M^{\varphi^+}_3(K_1)=1$ as the default values.

\section{Preliminary Results}

Note that there is no direct relationship between the degree $d(v_i)$ of a vertex $v_i$ and its colour $c(v_i)=c_j$ in $G$. Then, the following are some obvious observations on the three Zagreb indices of the basic graphs $K_1$ (or $P_1$), $K_2$ (or $P_2$), $P_3$ and $K_3$ (or $C_3$).

\begin{center}
\begin{enumerate}\itemsep0mm
\item[(i)] $M^{\varphi^+}_1(K_1)=M^{\varphi^-}_1(K_1)=1>0=M_1(K_1)$.
\item[(ii)] $M^{\varphi^+}_1(K_2) = M^{\varphi^-}_1(K_2)=5>2=M_1(K_2)$. 
\item[(iii)] $M^{\varphi^+}_1(P_3)=9>6= M_1(P_3)$ and $M^{\varphi^-}_1(P_3)=6=M_1(P_3)$.
\item[(iv)] $M^{\varphi^+}_1(K_3)=M^{\varphi^-}_1(K_3)=14>12=M_1(K_3)$.
\item[(v)] $M^{\varphi^+}_2(K_1)=M^{\varphi^-}_2(K_1)=0=M_2(K_1)$.
\item[(vi)] $M^{\varphi^+}_2(K_2)= M^{\varphi^-}_2(K_2)=2=M_2(K_2)$. 
\item[(vii)] $M^{\varphi^+}_2(P_3)=M^{\varphi^-}_2(P_3)=4=M_2(P_3)$.
\item[(viii)] $M^{\varphi^+}_2(K_3)=M^{\varphi^-}_2(K_3)=11<12=M_2(K_3)$. 
\item[(ix)] $M^{\varphi^+}_3(K_1)=M^{\varphi^-}_3(K_1)=1>0=M_3(K_1)$.
\item[(x)] $M^{\varphi^+}_3(K_2)=M^{\varphi^-}_3(K_2)=1>0=M_3(K_2)$.
\item[(xi)] $M^{\varphi^+}_3(P_3)=M^{\varphi^-}_3(P_3)=2=M_3(P_3)$.
\item[(xii)] $M^{\varphi^+}_3(K_3)=M^{\varphi^-}_3(K_3)=4>0=M_3(K_3)$. 
\end{enumerate}
\end{center}

The graphs of order $n\ge 4$ with maximum degree for all vertices and maximum number of edges and with maximum chromatic number are complete graphs. Therefore, some extremal bounds can be found for all graphs of order $n \ge 4$.

\begin{prop} 
For complete graphs $K_n,\ n\ge 4$ and for all $\varphi_t$ we have 
\begin{enumerate}\itemsep0mm
\item[(i)] $M^{\varphi_t}_1(K_n) < M_1(K_n)$.
\item[(ii)] $M^{\varphi_t}_2(K_n) < M_2(K_n)$.
\item[(iii)] $M^{\varphi_t}_3(K_n) > M_3(K_n)$.
\end{enumerate}
\end{prop} 
\begin{proof}
\textit{Part (i):} We have $M_1(K_n) = n(n-1)^2$ and $M^{\varphi_t}_1(K_n) = \sum\limits_{i=1}^{n}i^2 = \frac{1}{6}n(n+1)(2n+1)$ for all $\varphi_t$. Considering $x\in \R$ we need to find the positive root of the equation $x(x-1)^2 - \frac{1}{6}x(x+1)(2x+1)=0$.
Then, 
\begin{eqnarray*}
x(x-1)^2 - \frac{1}{6}x(x+1)(2x+1) & = & 0\\
\implies x(4x^2 - 15x + 5) & = & 0.\\
\end{eqnarray*}

\vspace{-0.6cm}

By approximation, positive $x \in (3.375; 3.5) \implies n\ge 4,\ n\in \N$. Hence, the result $M^{\varphi_t}_1(K_n) < M_1(K_n),\ n\ge 4$, for all colourings $\varphi_t$. 

\textit{Part (ii):} $M^{\varphi_t}_2(K_4)=\sum\limits_{j=1}^{3}\sum\limits_{i=2}^{4}j\cdot i=35< \frac{1}{2}\cdot4(4-1)^3=54=M_2(K_4)$.  Hence, the result holds for $K_n;\ n=4$.  We now proceed by mathematical induction. Assume the result holds for all $4\le n\le k$.  This implies that for $n=k$, we have $\sum\limits_{j=1}^{k-1}\sum\limits_{i=2}^{k}j\cdot i<\frac{1}{2}k(k-1)^3$, for all colourings $\varphi_t$. Now, consider the complete graph of order $n= k+1$.  Now, we have 

\begin{eqnarray*}
M^{\varphi_t}_2(K_{k+1}) & = & \sum\limits_{j=1}^{k}\sum\limits_{i=2}^{k+1}j\cdot i \\
& = & \sum\limits_{j=1}^{k-1}\sum\limits_{i=2}^{k}j\cdot i + (k+1)\sum\limits_{j=1}^{k}j \\
& = & M^{\varphi_t}_2(K_k) + \frac{1}{2}k(k+1)^2.
\end{eqnarray*}
 and 
\begin{eqnarray*}
M_2(K_{k+1}) & = & \frac{1}{2}(k+1)((k+1)-1)^3\\ 
& = & \frac{1}{2}(k+1)((k-1)+1)^3\\ 
& = & M_2(K_k) + \frac{1}{2}k(4k^2-3k+1).
\end{eqnarray*}

Hence, we must determine the value of $k$ for which the condition $\frac{1}{2}k(4k^2-3k+1)> \frac{1}{2}k(k+1)^2$ holds. For this, we determine the positive root of the equation $\frac{1}{2}x(4x^2-3x+1)- \frac{1}{2}x(x+1)^2= 0$, for real $x$ as follows.

\begin{eqnarray*}
\frac{1}{2}x(4x^2-3x+1)- \frac{1}{2}x(x+1)^2 & = & 0\\
\implies x^2(3x- 5) & = & 0\\
\implies x\in \left\lbrace 0, \frac{5}{3}\right\rbrace.
\end{eqnarray*}

Therefore,  $k\ge 4, k \in \N$ is valid and hence the result holds for $n=k+1$. Hence, by induction, the general result $M^{\varphi_t}_2(K_n)<M_2(K_n),\ n \ge 4$, for all colourings $\varphi_t$ follows. Also, note that we have already seen that the result is true for $n\le 3$.

\textit{Part (iii)}: Since $M_3(K_n)=0$ for all $n\ge 1$, the result is trivial.
\end{proof}

\ni The following theorem is the main result of this section. 

\begin{thm}\label{Thm-2.2}
For any simple connected graph $G$ of order $n\ge 4$ and for all colourings $\varphi_t$, we have
\begin{enumerate}\itemsep0mm
\item[(i)] $M^{\varphi_t}_1(G) < M^{\varphi_t}_1(K_n)$. 
\item[(ii)] $M^{\varphi_t}_2(G) < M^{\varphi_t}_2(K_n)$. 
\item[(iii)] $M^{\varphi_t}_3(G) < M^{\varphi_t}_3(K_n)$.
\end{enumerate}  
\end{thm} 
\begin{proof}
\textit{Part (i):} Let the minimum parameter set of colours of $K_n$ be $\mathcal{C}$, $|\mathcal{C}|=n$.  Consider any edge $e=uv\in E(K_n)$. Clearly, the minimum parameter set of colours for $G-e$ is given by $\mathcal{C}'$, where $|\mathcal{C}'|=n-1$. Since, up to isomorphism, all colourings $\varphi_t$ in $K_n-e$ are equivalent proper minimum parameter colourings, the value $M^{\varphi^-}_1(K_n-e)$ is obtained for $\varphi:u\mapsto c_1$, $\varphi:v\mapsto c_1$ and $\varphi:w\mapsto c_j$, $c_j\in \mathcal{C}'-\{c_1\}$, for all $w \in V(K_n-e)-\{u,v\}$.  Similarly, the value $M^{\varphi^+}_1(K_n-e)$ is obtained for $\varphi:u\mapsto c_{n-1}$, $\varphi:v\mapsto c_{n-1}$ and $\varphi:w\mapsto c_j$, $c_j\in \mathcal{C}' -\{c_{n-1}\} \forall\, w \in V(K_n-e)-\{u,v\}$.  Furthermore, $M^{\varphi^-}_1(K_n-e) = \sum\limits_{i=1}^{n-1}i^2 + 1^2<\sum\limits_{i=1}^{n-1}i^2 + n^2 = \sum\limits_{i=1}^{n}i^2 = M^{\varphi_t}_1(K_n)$. Similarly, $M^{\varphi^+}_1(K_n-e) = \sum\limits_{i=1}^{n-1}i^2 + (n-1)^2 < \sum\limits_{i=1}^{n-1}i^2 + n^2 = \sum\limits_{i=1}^{n}i^2 = M^{\varphi_t}_1(K_n)$.  Hence, $ M^{\varphi_t}_1(K_n-e) < M^{\varphi_t}_1(K_n)$ for all $\varphi_t$.  After removal of the appropriate edges from $K_n$ to obtain $G$ (which is always possible), the result follows.

\textit{Part (ii):} The result follows through similar reasoning as in Part (i). As $\frac{n}{2}(n-1)^2>\frac{n}{2}(n-1)$ for all $n\ge 3$, after appropriate minimum parameter colouring, corresponding to $M^{\varphi^+}_2(K_n-uv)$
we have $\frac{n}{2}(n-1)^2>\frac{n}{2}(n-1)$ for all $n\ge 4$. 

\textit{Part (iii):} The result follows through similar reasoning as in Part (i). We have $M^{\varphi_t}_3(K_n) = \sum\limits_{j=1}^{n-1}\sum\limits_{i=1}^{j}i = \sum\limits_{j=1}^{n-1}\frac{1}{2}j(j+1) = \frac{\ell}{6}(\ell^2 + 3\ell +2),\ \ell=n-1$, for all $\varphi_t$.  After appropriate proper minimum parameter colouring, we must show that corresponding to $M^{\varphi^-}_3(K_n-uv)$ and $M^{\varphi^+}_3(K_n-uv)$, the inequality $\sum\limits_{i=1}^{n-2}i<\sum\limits_{i=1}^{n-1}i$ holds and which is trivially true. Hence, after removal of the appropriate edges from $K_n$ to obtain $G$ (which is always possible), the result follows.
\end{proof}

\ni The following result is an immediate consequence of Theorem \ref{Thm-2.2}.

\begin{cor} 
If $G'$ is a subgraph of $G$ then we have
\begin{enumerate}\itemsep0mm
\item[(i)] $M^{\varphi_t}_1(G') < M^{\varphi_t}_1(G)$.
\item[(ii)] $M^{\varphi_t}_2(G') < M^{\varphi_t}_2(G)$.
\item[(iii)] $M^{\varphi_t}_3(G') < M^{\varphi_t}_3(G)$.
\end{enumerate}  
\end{cor} 
\begin{proof}
Immediate from the proof of Theorem \ref{Thm-2.2}.
\end{proof}

\section{Chromatic Zagreb Indices of Certain Graphs}

Since chromatic Zagreb indices are dependent on a proper minimum parameter colouring of a graph $G$ as well as the edges in the case of $M^{\varphi_t}_2(G)$ and $M^{\varphi_t}_3(G)$, the introductory study of graphs with known chromatic number and well-defined edge set is important. For trees we have the first general result.

\begin{thm}
For a finite tree of order $n\ge 4$, we have
\begin{enumerate}\itemsep0mm
\item[(i)] $n+3\le M^{\varphi^-}_1(T) \le M^{\varphi^+}_1(T)\le 4n- 3$.
\item[(ii)] $M^{\varphi^-}_2(T)=M^{\varphi^+}_2(T)=2(n-1)$.
\item[(iii)] $M^{\varphi^-}_3(T)=M^{\varphi^+}_3(T)= n-1$.
\end{enumerate}  
\end{thm} 
\begin{proof}
\textit{Part (i):} Since $\chi(T)=2$ and the star $K_{1,n-1}$ is the tree of order $n$ with maximum pendant vertices (leafs), it is obvious that $M^{\varphi^-}_1(K_{1,n-1})$ and $M^{\varphi^+}_1(K_{1,n-1})$ are the lower and upper bounds for the first chromatic Zagreb index of trees of order $n$.  Therefore, the result follows.

\textit{Part (ii):} A tree $T$ of order $n$ has exactly $n-1$ edges. For each edge $uv \in E(T)$, a term of the sum to determine $M^{\varphi^-}_2(T)$ is $2$. Hence the result.

\textit{Part (iii):} For each edge $uv \in E(T)$ an absolute term of the sum to determine $M^{\varphi^-}_3(T)$ is $1$. Hence the result.
\end{proof}

For a $r$-partite graph $K_{n_1,n_2,n_3,\dots,n_r}$, $r\ge 2$, and $n_1\le n_2\le n_3 \le \cdots \le n_r$ the next lemma follows directly from the respective definitions.

\begin{lem}\label{Lem-3.2}
For the $r$-partite graph $K_{n_1,n_2,n_3,\ldots,n_r},\ r\ge 2$ and $n_1\le n_2\le n_3 \le \ldots \le n_r$, we have
\begin{enumerate}\itemsep0mm 	 
\item[(i)] $M^{\varphi^+}_1(K_{n_1,n_2,n_3,\ldots,n_r})= \sum\limits_{i=1}^{r}n_i\cdot i^2
\text{and} \\
 M^{\varphi^-}_1(K_{n_1,n_2,n_3,\dots,n_r}) = \sum\limits_{i=0}^{r-1}n_{i+1}\cdot (r-i)^2.$
\item[(ii)] $M^{\varphi^+}_2(K_{n_1,n_2,n_3,\dots,n_r}) =   \sum\limits_{i=1}^{r-1}\sum\limits_{j=i+1}^{r}n_i n_j(i\cdot j) \text{and}\\
M^{\varphi^-}_2(K_{n_1,n_2,n_3,\dots,n_r})  = \sum\limits_{i=1}^{r-1}\sum\limits_{j=i+1}^{r}n_i n_j(r-i)\cdot (r-j)$. 
\item[(iii)] $M^{\varphi^+}_3(K_{n_1,n_2,n_3,\dots,n_r}) = \sum\limits_{i=1}^{r-1}\sum\limits_{j=i+1}^{r}n_i n_j(j-i) = M^{\varphi^-}_3(K_{n_1,n_2,n_3,\dots,n_r})$,

$\therefore M^{\varphi_t}_3(K_{n_1,n_2,n_3,\dots,n_r}) = \sum\limits_{i=1}^{r-1}\sum\limits_{j=i+1}^{r}n_i n_j(j-i) $. 
\end{enumerate}
\end{lem} 

Because $n_i\in \N$ is arbitrary, any closed formulae do not exist for the results provided in Lemma \ref{Lem-3.2}. Instead, simplified sum formulae exist for the cases $n_i=n$ for all $1\le i\le r$. 

\begin{prop} 
For any complete $r$-partite graphs with $r\ge 2$ we have
\begin{enumerate}\itemsep0mm 
\item[(i)]  $M^{\varphi_t}_1(K_{\underbrace{n,n,n,\dots,n}_{r-entries}}) = \frac{n}{6}r(r+1)(2r+1)$. 
\item[(ii)]  $M^{\varphi_t}_2(K_{\underbrace{n,n,n,\dots,n}_{r-entries}}) =\frac{n^2}{2}\sum\limits_{i=2}^{r}i^2(i-1)$.
\item[(iii)]  $M^{\varphi_t}_3(K_{\underbrace{n,n,n,\dots,n}_{r-entries}}) = n^2\sum\limits_{i=1}^{r-1}i(r-1)$. 
\end{enumerate}
\end{prop} 
\begin{proof}
\textit{Part (i):} It easily follows that, $M^{\varphi_t}_1(K_{\underbrace{n,n,n,\dots,n}_{r-entries}})=n\cdot M^{\varphi_t}_1(K_r)$.  Hence, $M^{\varphi_t}_1(K_{\underbrace{n,n,n,\dots,n}_{r-entries}}) = \frac{n}{6}r(r+1)(2r+1)$.

\textit{Part (ii):} Let $G$ be an $r$-partite graph. Then, it follows that the vertex set can be partitioned into $r$ non-empty subsets of non-adjacent vertices. Let $V(K_{\underbrace{n,n,n,\dots,n}_{r-entries}})=\bigcup\limits_{i=1}^{r}V_i$ and $|V_i|= n,\ \forall\, i$.  Hence, exactly $\frac{1}{2}r(r-1)$ distinct complete bipartite subgraphs $\langle V_i\cup V_j\rangle;\ 1\le i<j\le r$ are needed to account for all edges. Consider the minimum parameter colouring $\varphi$, with respect to which, we have $V_1\mapsto c_1,V_2\mapsto c_2, V_3\mapsto c_3,\ldots, V_r\mapsto c_r$.  Determining the partial sum terms sequentially for $1\le i<j\le r$, we have $n(\underbrace{1\cdot 2+1\cdot 2+\ldots +1\cdot 2}_{n-entries} + \underbrace{1\cdot 3+1\cdot 3+\ldots +1\cdot 3}_{n-entries}+\cdots + \underbrace{1\cdot r+1\cdot r+\ldots +1\cdot r}_{n-entries})+
n(\underbrace{2\cdot 3+2\cdot 3+\ldots +2\cdot 3}_{n-entries}+\ldots + \underbrace{2\cdot r+2\cdot r+\ldots +2\cdot r}_{n-entries})+ 
\ldots\ldots\ldots\ldots\ldots\ldots\ldots+
n(\underbrace{(r-1)\cdot r+(r-1)\cdot r+(r-1)\cdot r+\ldots +(r-1)\cdot r}_{n-entries})$. 

The total sum is simplified to $M^{\varphi}_2(K_{\underbrace{n,n,n,\dots,n}_{r-entries}}) =\frac{n^2}{2}\sum\limits_{i=2}^{r}i^2(i-1)$. As all possible minimum parameter colourings are equivalent, it can be noted that the general result $M^{\varphi_t}_2(K_{\underbrace{n,n,n,\dots,n}_{r-entries}}) =\frac{n^2}{2}\sum\limits_{i=2}^{r}i^2(i-1)$ follows.

Part (iii): The proof is similar to that in Part (ii). Note that the partial sum terms derived from the edges in $\langle V_i\cup V_j\rangle$ are $n\sum\limits_{\ell=1}^{n}|i-j|$ or $n\sum\limits_{\ell=1}^{n}(j-i)$. 
\end{proof}

\subsection{Chromatic Zagreb Indices of $m$-regular Thorn Graph}

Different properties of a path with pendant vertices (varying in number), attached to the path vertices were studied first in \cite{HS1} and in a series of papers following it. These graphs were later named \textit{caterpillars}. Later, as a generalisation of caterpillars, another type of graphs called \textit{thorn graphs} has been defined in \cite{IG1} as a graph, denoted by $G^\star$, obtained from a graph $G$ of order $n$ by attaching $p_i\ge 0,\ i=1,2,\ldots,n$, pendant vertices to the $i^{th}$ vertex of $G$. In this section, $p_i = m,\ m \in \N_0$ for all $i$.  Generally, if an edgeless graph exists, a \textit{thornless} thorn graph also exists (see \cite{IG1}). As a thorn graph is defined for $m=0$, the graph $G$ itself is a \textit{thornless} thorn graph, denoted by $G^\star_{m=0}$.

If the minimum parameter colour set $\mathcal{C} = \{c_1,c_2,c_3,\dots,c_\ell\}$ allows a proper colouring of $G$, then it follows that $M^{\varphi^-}_1(G) = \sum\limits_{i=1}^{\ell}\theta(c_i)\cdot i^2$, $c_i \in \mathcal{C}$ if $\theta(c_1)\ge \theta(c_2)\ge \cdots \ge \theta(c_\ell)$.  For the inverse map $f:c_i\mapsto c_{(\ell+1)-i}$, $1\le i \le \ell$ the value $M^{\varphi^+}_1(G)=\sum\limits_{i=1}^{\ell}\theta(c_i))\cdot(\ell+1- i)^2$ is obtained. These observations lead to the following general result.

\begin{thm} 
For a graph $G$ of order $n$, we have 
\begin{enumerate}\itemsep0mm
	\item[(i)] If $\theta(c_1)\ge \theta(c_2)\ge \ldots \ge \theta(c_\ell)$, where each $c_i \in \mathcal{C}$ allows $M^{\varphi^-}_1(G)$, then $M^{\varphi^-}_1(G^\star)= M^{\varphi^-}_1(G)+ 4m\cdot\theta(c_1) + m(n-\theta(c_1))$.
	\item[(ii)] For the map $f:c_i\mapsto c_{(\ell+1)-i}$, $1\le i \le \ell$; $M^{\varphi^+}_1(G^\star)=M^{\varphi^+}_1(G)+ m(\ell-1)^2\cdot\theta(c_1)+ m\ell^2(n-\theta(c_1))$.
	\item[(iii)] If $\theta'(c_1)\ge \theta'(c_2)\ge \ldots \ge \theta'(c_\ell)$, where each $c_i \in \mathcal{C}$ allows $M^{\varphi^-}_2(G)$, then $M^{\varphi^-}_2(G^\star)=M^{\varphi^-}_2(G)+ 2m\cdot\theta'(c_1) + \sum\limits_{c_i \in \mathcal{C}-\{c_1\}}mi\cdot\theta(c_i)$. 
	\item[(iv)] For the map $f:c_i\mapsto c_{(\ell+1)-i},\ 1\le i \le \ell$; $M^{\varphi^+}_2(G^\star)=M^{\varphi^+}_2(G)+m\ell(\ell-1)\cdot\theta'(c_1) + \sum\limits_{c_i \in \mathcal{C}-\{c_1\}}m\ell(\ell+1-i)\cdot\theta'(c_i)$. 
	\item[(v)] If $\theta''(c_1)\ge \theta''(c_2)\ge \ldots \ge \theta''(c_\ell)$, where each $c_i \in \mathcal{C}$ allows $M^{\varphi^-}_3(G)$, then $M^{\varphi^-}_3(G^\star) = M^{\varphi^-}_3(G)+ mn$. 
	\item[(vi)] For the map $f:c_i\mapsto c_{(\ell+1)-i}$, $1\le i \le \ell$, we have $M^{\varphi^+}_3(G^\star)=M^{\varphi^+}_3(G) + \lfloor \frac{\ell}{2}\rfloor \cdot \theta''(c_{\lceil \frac{\ell}{2}\rceil}) + \sum\limits_{i=1}^{\lfloor\frac{\ell}{2}\rfloor}(\ell-i)\cdot \theta''(c_i)+ \sum\limits_{i=\lceil \frac{\ell}{2}\rceil}^{\ell-1}i\cdot \theta''(c_{i+1})$  if $\ell$ is odd and $M^{\varphi^+}_3(G^\star)=M^{\varphi^+}_3(G) + \sum\limits_{i=1}^{\frac{\ell}{2}}(\ell-i)\cdot \theta''(c_i) + \sum\limits_{i=\frac{\ell}{2}}^{\ell-1}i\cdot \theta''(c_{i+1})$, if $\ell$ is even.
\end{enumerate}
\end{thm} 
\begin{proof}
\textit{Part (i):} Clearly, we have $M^{\varphi^-}(G^\star_{m=0}) = M^{\varphi^-}_1(G^\star)$.  For $m\ge 1$, colour the pendant vertices attached to the $\theta(c_1)_G$ vertices by the colour $c_2$.  Colour all the pendant attached to the other $(n-\theta(c_1))_G$ vertices by the colour $c_1$.  From the definition of $M^{\varphi^-}_1(G^\star)$ it follows that we add the minimum sum terms $m\cdot \theta(c_1)\cdot 2^2$ as well as $m(n-\theta(c_1))\cdot1^2$. Therefore the result follows.

\textit{Part (ii):} The result follows similar to Part (i). Note that the $\theta(c_1)$ vertices in $G$ now coloured $c_\ell$ will have pendant vertices all coloured $c_{\ell-1}$. All other vertices in $G$ will have all pendant vertices coloured $c_\ell$. 

\textit{Part (iii):} For a minimum parameter colouring the same as Part (i) and applying the definition of $M^{\varphi^-}_2(G^\star)$, the result follows.

\textit{Part (iv):} For a minimum parameter colouring the same as Part (ii) and applying the definition of $M^{\varphi^+}_2(G^\star)$, the result follows.

\textit{Part (v):} For the vertices in $G$ coloured $c_1$ the $m$ pendant vertices can be coloured $c_2$.  Similarly, for the vertices in $G$ coloured $c_\ell$ the pendant vertices can be coloured $c_{\ell-1}$. For the vertices in $G$ coloured $c_j$, $2\le j\le \ell-1$ the pendant vertices can be coloured $c_{j-1}$ or $c_{j+1}$. Therefore the sum terms resulting from the thorn-edges are all equal to 1. The latter implies the result.

\begin{enumerate}\itemsep0mm
\item[(a)] Let $\ell$ be odd. The unique \textit{mid-colour} is $c_{\lceil\frac{\ell}{2}\rceil}$. Colour all the pendant vertices of vertices in $G$ having mid-colouring $c_{\lceil\frac{\ell}{2}\rceil}$, the colour $c_1$ or $c_\ell$ to ensure the maximum sum term $\theta''(c_{\lceil \frac{\ell}{2}\rceil})\cdot|\lceil\frac{\ell}{2}\rceil - 1|$ or $\theta''(c_{\lceil \frac{\ell}{2}\rceil})\cdot|\lceil\frac{\ell}{2}\rceil - \ell|$.  Colour all the pendant vertices of the vertices $\theta''(c_i)$, $1\le i \le \lfloor\frac{\ell}{2}\rfloor$ which are now carry inverse colouring, with the colour $c_\ell$ to ensure maximum sum terms in respect of the definition of $M^{\varphi^+}_3(G^\star)$.  Also colour all the pendant vertices of the vertices $\theta''(c_i)$, $\lceil\frac{\ell}{2}\rceil\le i \le \ell-1 $ which now have inverse colouring, with the colour $c_1$ to ensure maximum sum terms in respect of the definition of $M^{\varphi^+}_3(G^\star)$.  Through total summation the result follows.

\item[(b)] Let $\ell$ be even. The proof is similar to (a) noting that no unique \textit{mid-colour} exists.
\end{enumerate}
\vspace{-0.75cm}
\end{proof}

\section{Chromatically Stable Graphs}

Generally, a graph $G$ is said to be \textit{critical} with respect to an invariant if the removal of any element changes the invariant. Note that $M^{\varphi_t}_1(G)$ is \textit{vertex critical}. Furthermore, both $M^{\varphi_t}_2(G)$ and $M^{\varphi_t}_3(G)$ are critical indices (\textit{element critical}). Hence, for a connected graph $G$ of order $n$ and a spanning tree $T$ of $G$, we have minimal $M^{\varphi_t}_2(T),$ $M^{\varphi_t}_3(T)$. That is,  $M^{\varphi_t}_2(T)\le M^{\varphi_t}_2(G)$ and $M^{\varphi_t}_3(T)\le M^{\varphi_t}_3(G)$.  

\ni The following theorem is perhaps stating the obvious, but important to formalise.

\begin{thm}
Consider finite positive integer $n$. Let $\mathfrak{G}$ denotes the set of all unlabeled, connected, simple graphs of order $n$, $\mathcal{T}_G$ be the set of all unlabeled spanning trees of a given graph $G\in \mathfrak{G}$ and $\mathfrak{T}$ be the set of all unlabeled trees of order $n$. Then, $\mathfrak{T}=\bigcup\limits_{G\in \mathfrak{G}}\mathcal{T}_G$.
\end{thm}
\begin{proof}
If $n=1$. the result is trivially true. Hence, assume $n\ge 2$. The number of labeled, connected, simple graphs on finite $n$ vertices is given by $n\cdot2^{\binom{n}{2}}$ (\cite{HP1}). Therefore, the number of non-isomorphic (unlabeled) connected, simple graphs of order $n$ is finite hence, $|\mathfrak{G}| < n\cdot2^{\binom{n}{2}}$ is finite. It is known that a connected, simple graph $G$ of order, finite $n \in \N$ has at least one spanning tree. As $n$ is finite, by Cayley's theorem on the number of distinct spanning trees of a connected graph (see \cite{BM1}), the number of labeled spanning trees of any graph $G\in\mathfrak{G}$ is finite and is equal to $|\mathcal{T}_G| \leq n^{n-2}$. This implies that $|\mathfrak{T}|$ is finite. 
	
Assume that there exists a tree $T\in \mathfrak{T}$ which is not a spanning tree of some $G\in \mathfrak{G}$. Add at least one edge between two non-adjacent vertices of $T$ to obtain the graph $G'$. Clearly, $G'$ is also an unlabeled connected, simple graph and hence $G'\in \mathfrak{G}$, which is a contradiction to the hypothesis and hence we have
$$ T \in \mathcal{T}_{G'} \implies T\in \bigcup\limits_{G'\in \mathfrak{G}}\mathcal{T}_{G'} \implies\mathfrak{T} =\bigcup\limits_{G\in \mathfrak{G}}\mathcal{T}_G.$$

\vspace{-1cm}
	
\end{proof}

The implication of the above theorem is that the sets $\mathfrak{G}$, $\mathcal{T}_G$, $\mathfrak{T}$ are well-defined. Through mathematical induction the theorem can be proven for $n \to \infty$.

\begin{thm}
For any unlabeled tree $T$ of order $n$, we have
\begin{enumerate}\itemsep0mm
\item[(i)] $M^{\varphi_t}_2(T) = \min\{ M^{\varphi_t}_2(G): G \in \mathfrak{G}\}$.
\item[(ii)] $M^{\varphi_t}_3(T) = \min\{ M^{\varphi_t}_3(G): G \in \mathfrak{G}\}$.
\end{enumerate}
\end{thm}
\begin{proof}
\textit{Part (i)} For each edge $e$ of $T$, one end vertex is coloured $c_1$ and the other end vertex coloured $c_2$. Hence $\forall e \in E(T)$ the product $1\cdot2$ is constant and minimum over all proper colourings of any  $G \in \mathfrak{G}$. Therefore, for any edge $v_iv_j \in E(G)$, $M^{\varphi_t}_2(G)=\sum \limits_{i=1}^{n-1}\sum\limits_{j=2}^{n} (s\cdot k)=\min\{M^{\varphi_t}_2(G): G\in \mathfrak{G}\}, c(v_i)= s, c(v_j)=k, 1\le t \le \ell\,!$,

\textit{Part (ii):} The proof follows similar to that of Part (i) except now the term $|1-2|=|2-1|$ applies. Hence, for $v_iv_j \in E(G)$, $M^{\varphi_t}_3(G)=\sum\limits_{i=1}^{n-1}\sum\limits_{j=2}^{n} |s-k|=\min\{ M^{\varphi_t}_3(G): G \in \mathfrak{G}\}, c(v_i)= s,\ c(v_j)=k, 1\le t \le \ell\,!$. 
\end{proof}

\begin{defn}{\rm 
A graph $G$ is said to be \textit{chromatically stable}, if the addition of a new edge $e$ to $G$ does not result in a change in the chromatic number of the new graph. That is, $\chi(G+e)=\chi(G)$. If the addition of a new edge changes to $G$ changes the chromatic number, then $G$ is called a  \textit{chromatically unstable} graph.}
\end{defn}
 
Some examples of chromatically unstable graphs are the star graph $K_{1,n}$ and a graph $K_n-e$, $e\in E(K_n)$.  The complete graph $K_n$ is said to be chromatic perfectly stable.  We know that the cycles $C_n,\ n\ge 4$, the star graphs and the complete bipartite graphs $K_{m,n}$ are some of the $2$-chromatic graphs which are chromatically unstable.

The following theorem discusses a necessary and sufficient condition for a $2$-chromatic graph to be a chromatically stable graph.

\begin{thm}\label{Thm-4.1}
A $2$-chromatic graph $G$ is chromatically stable if and only if $G$ is not isomorphic to a complete bipartite graph.
\end{thm} 
\begin{proof}
Let $G$ be a chromatically stable graph with $\chi(G)=2$. Hence, $G$ is a bipartite graph. Let $(V_1,V_2)$ be a bipartition of $G$. Let $\varphi: V_1\mapsto c_1$ and $\varphi: V_2\mapsto c_2$. Let $uv$ be an edge not in $E(G)$. If $u$ and $v$ are in the same partition $V_1$ or $V_2$, then $G+uv$ contains a triangle and hence we have $\chi(G+uv)=3$, which contradicts the hypothesis that $G$ is chromatically stable.  Therefore, one end vertex of $uv$ must be in $V_1$ and the other in $V_2$. If $u\in V_1$, $v\in V_2$, then $G+uv$ is also a bipartite graph and hence $\chi(G+uv)=2$. Hence, $G$ is not (isomorphic to) a complete bipartite graph.

If possible, assume that $G$ is chromatically unstable. Then, for any edge $e=uv\notin E(G)$, we have 
$\chi(G+uv)>\chi(G)$. Then, the end vertices of the edge $e$ must be in the same partition. That is, $G$ isomorphic to a complete bipartite graph. 
\end{proof}

The concept of the $\chi$-stability number of a chromatically stable graph can be introduced as follows. 

\begin{defn}{\rm 
The \textit{$\chi$-stability number} of a graph $G$, denoted by $\varrho(G)$, is defined to be the minimum number of edges that can be added to a chromatically stable graph so as to make it a chromatically unstable graph.}
\end{defn}

\begin{prop} 
If $G$ is a $2$-chromatic graph $G$ of size $\epsilon$, we have  $\varrho(G) = \theta(c_1)\theta(c_2)-\epsilon$.
\end{prop} 
\begin{proof}
Since $G$ is a $2$-chromatic graph, then $G$ is a bipartite graph with bipartition $(V_1, V_2)$ such that $V_1$ is the colour class of $c_1$ and $V_2$ is the colour class of the colour $c_2$. Therefore, $\theta(c_1)=|V_1|$ and $\theta(c_2)=|V_2|$. Clearly, $\theta(c_1)\theta(c_2)-\epsilon$ is the number of edges to be added to $G$ to obtain a complete bipartite graph $K_{|V_1|,|V_2|}$, which is chromatically unstable, by Theorem \ref{Thm-4.1}. Therefore, $\varrho(G) = \theta(c_1)\theta(c_2)-\epsilon$.
\end{proof}

\section{Conclusion}

In this paper, the study of irregularity as defined by the three well known Zagreb indices, but based on a vertex property through colouring has been introduced. Many results may follow from studying other known graphs with well-defined edge set $E(G)$. 

Determining these parameters for various other graph classes offers much scope for further investigations. The relations of chromatic Zagreb indices of various operations, products and powers of graphs with the of chromatic Zagreb indices individual graphs are also worthy for future studies. Another area, we suggest for future studies is to determine chromatic Zagreb indices of certain associated graphs such as the line graphs, total graphs, homeomorphic graphs etc. of  given graphs.

A good understanding of chromatic stable graphs will be worthy for future studies on the corresponding chromatic Zagreb indices. The characterisation of chromatically stable graphs $G$ with a fixed chromatic number $\chi(G)=k>2$ remains open further studies. The relationship between a chromatically stable graph $G$ and the chromatic stability of its complement $\bar{G}$, if exists, are also worthy for future studies.

\end{document}